\documentclass[11pt,english,reqno]{amsart}
\usepackage[T1]{fontenc}
\usepackage[latin9]{inputenc}
\usepackage{amstext}
\usepackage{amsthm}

\makeatletter
\usepackage{amsmath,amsfonts,amssymb,amsthm,epsfig}

\usepackage{color}
\usepackage[final,allcolors=blue,colorlinks=true]{hyperref}

\def\R{\mathbb R}
\def\N{\mathbb N}

\def\ga{\gamma}
\def\de{\delta}
\def\ep{\epsilon}

\def\si{\sigma}
\def\ta{\theta}
\def\var{\varphi}
\def\na{\nabla}
\def\Om{\Omega}  
\def\De{\Delta}      

\def\cal{\mathcal}
\def\L{\mathcal L}                                       
\def\wq{\infty}
\def\pa{\partial}
\def\loc{\text{\rm loc}}

\newcommand{\D}{{\rm d}}
\newcommand{\wto}{\rightharpoonup}                

\numberwithin{equation}{section}
\newtheorem{theorem}{Theorem}[section]

\newtheorem{definition}[theorem]{Definition}

\newtheorem{lemma}[theorem]{Lemma}

\newtheorem{proposition}[theorem]{Proposition}

\theoremstyle{definition}

\makeatother

\usepackage{babel}
\begin{document}
\title[Multi-peak solutions]{Multi-peak positive solutions to a class of  Kirchhoff equations}

     \author[P. Luo,  S. Peng, C. Wang and C.-L. Xiang]{Peng Luo, Shuangjie Peng, Chunhua Wang and  Chang-Lin Xiang}

\address[Peng Luo]{ School of Mathematics and Statistics and Hubei Key Laboratory of Mathematical Sciences,  Central China Normal University, Wuhan,  430079, P.R. China}
\email[]{luopeng@whu.edu.cn}
\address[Shuangjie Peng]{School of Mathematics and Statistics and Hubei Key Laboratory of Mathematical Sciences, Central China Normal University,  Wuhan,  430079, P.R. China}
\email[]{sjpeng@mail.ccnu.edu.cn}
\address[Chunhua Wang]{School of Mathematics and Statistics and Hubei Key Laboratory of Mathematical Sciences, Central China Normal University,  Wuhan,  430079, P.R. China}
\email[]{chunhuawang@mail.ccnu.edu.cn}
\address[Chang-Lin Xiang]{School of Information and Mathematics, Yangtze University, Jingzhou 434023, P.R. China}
\email[]{changlin.xiang@yangtzeu.edu.cn}
\thanks{Luo and Wang are partially supported by  self-determined research funds of CCNU from colleges' basic research and operation of MOE(CCNU16A05011, CCNU17QN0008). Peng and Wang are also financially supported by NSFC (No. 11571130, No.11671162). Xiang is  financially supported by  the Yangtze Youth Fund, No. 2016cqn56.}

\begin{abstract}
 In the present paper, we consider the  nonlocal Kirchhoff  problem
 \begin{eqnarray*} -\left(\epsilon^2a+\epsilon b\int_{\mathbb{R}^{3}}|\nabla u|^{2}\right)\Delta u+V(x)u=u^{p},\,\,\,u>0 &  & \text{in }\mathbb{R}^{3}, \end{eqnarray*} where $a,b>0$, $1<p<5$ are constants, $\epsilon>0$ is  a parameter.  Under some mild assumptions on the function $V$, we obtain multi-peak solutions for $\epsilon$ sufficiently small by Lyapunov-Schmidt reduction method. Even though many results on single peak solutions to singularly perturbed Kirchhoff problems have been derived in the literature by various methods, there exist no results on multi-peak solutions before this paper, due to some difficulties caused by the nonlocal term $\left(\int_{\mathbb{R}^3}|\nabla u|^2\right)\Delta u$. A remarkable  new feature of this problem is that  the corresponding unperturbed problem turns out to be a system of partial differential equations, but not a  single Kirchhoff equation, which is quite different from most of elliptic singular perturbation problems.
\end{abstract}


\maketitle

{\small
\smallskip
\newline
}
\bigskip

\section{Introduction and main result}

Let $a,b>0$ and $1<p<5$. In this paper, we are concerned with the
following singularly perturbed Kirchhoff problem
\begin{eqnarray}
-\left(\ep^{2}a+\ep b\int_{\R^{3}}|\na u|^{2}\right)\De u+V(x)u=u^{p}, & u>0 & \text{in }\R^{3},\label{eq: Kirchhoff}
\end{eqnarray}
where $\ep>0$ is a parameter, $V:\R^{3}\to\R$ is a bounded continuous
function.

Problem (\ref{eq: Kirchhoff}) and its variants have been studied
extensively in the literature. To extend the classical D'Alembert's
wave equations for free vibration of elastic strings, Kirchhoff \cite{Kirchhoff-1883}
proposed for the first time the following time dependent wave equation
\[
\rho\frac{\pa^{2}u}{\pa t^{2}}-\left(\frac{P_{0}}{h}+\frac{E}{2L}\int_{0}^{L}\left|\frac{\pa u}{\pa x}\right|^{2}\right)\frac{\pa^{2}u}{\pa x^{2}}=0.
\]
Bernstein \cite{Bernstein-1940} and Pohozaev \cite{Pohozaev-1975}
studied the above type of Kirchhoff equations quite early. Much attention
was received until J.L. Lions \cite{Lions-1978} introducing an abstract
functional framework to this problem. More interesting results can
be found in e.g. \cite{Arosio-Panizzi-1996,DAncona-Spagnolo-1992}
and the references therein. From a mathematical point of view, Kirchhoff
equations is nonlocal, in the sense that, the term $\left(\int|\na u|^{2}\D x\right)\De u$
depends not only on the pointwise value of $\De u$, but also on the
integral of $|\na u|^{2}$ over the whole space. This new feature
brings new mathematical difficulties that make the study of Kirchhoff
type equations particularly interesting. We refer to e.g. \cite{Perera-Zhang-2006}
and to e.g. \cite{Deng-Peng-Shuai-2015,Figueiredo et al-2014,Guo-2015,He-Li-2015,Li-Li-Shi-2012,Li-Ye-2014}
for mathematical researches on Kirchhoff type equations on bounded
domains and in the whole space, respectively.

Eq. (\ref{eq: Kirchhoff}) is also closely related to Schr\"odinger
equations. Indeed, when $b=0$, Eq. (\ref{eq: Kirchhoff}) reduces
to the problem
\begin{eqnarray*}
-\ep^{2}\De u+V(x)u=u^{p}, & u>0 & \text{in }\R^{3},
\end{eqnarray*}
which are special cases of the perturbed Schr\"odinger equations
\begin{eqnarray}
-\ep^{2}\De u+V(x)u=u^{q}, & u>0 & \text{in }\R^{n},\label{eq: Schrodinger equations}
\end{eqnarray}
where $1<q$ is subcritical and $n\ge1$. Flower and Weinstein \cite{Flower-Weinstein-1986},
Oh \cite{Oh-1988,Oh-1990}, del Pino and Felmer \cite{delPino-Felmer-1996,delPino-Felmer-1998},
Gui \cite{Gui-1996} and many others proved the existence of solutions
to Eq. (\ref{eq: Schrodinger equations}) for $\ep>0$ sufficiently
small (the so called semiclassical solutions). In particular, Oh \cite{Oh-1990}
obtained multi-peak solutions to problem (\ref{eq: Schrodinger equations})
by using the Lyapunov-Schmidt reduction method, and del Pino and Felmer
\cite{delPino-Felmer-1998}, Gui \cite{Gui-1996} obtained multi-peak
solutions to the above perturbed Schr\"odinger equations with more
general nonlinearity by variational methods, respectively. For main results on multi-peak solutions, see e.g. \cite{Cao-Noussair-Yan-2008,Cao-Peng-2009,Noussair-Yan-2000} and the references therein. We remark
that to construct multi-peak solutions, a common building block of
Flower and Weinstein \cite{Flower-Weinstein-1986}, Oh \cite{Oh-1988,Oh-1990}
is the unique positive radial solution in $H^{1}(\R^{n})$ of the
unperturbed Schr\"odinger equation
\begin{eqnarray}
-\De u+u=u^{q}, & u>0 & \text{in }\R^{n}.\label{eq: unperturbed Schrodinger equation}
\end{eqnarray}

Now we review some known results on Kirchhoff equations. It seems
that He and Zou \cite{He-Zou-2012} is the first to study singularly
perturbed Kirchhoff equations. In \cite{He-Zou-2012}, they considered
the problem
\begin{eqnarray*}
-\left(\ep^{2}a+\ep b\int_{\R^{3}}|\na u|^{2}\right)\De u+V(x)u=f(u), & u>0 & \text{in }\R^{3},
\end{eqnarray*}
where $V$ is assumed to satisfy the global condition of Rabinowitz
\cite{Rabinowitz-1992-ZAMP}
\begin{equation}
\liminf_{|x|\to\wq}V(x)>\inf_{x\in\R^{3}}V(x)>0,\label{eq: Rabinowitz condition}
\end{equation}
and $f:\R\to\R$ is a nonlinear function with subcritical growth of
type $u^{q}$ for some $3<q<5$. They proved the existence of multiple
positive solutions for $\ep$ sufficiently small. Among other results,
Wang et al. \cite{Wang-Tian-Xu-Zhang-2012} established
some existence and nonexistence  results for Kirchhoff equations with critical growth
\begin{eqnarray*}
-\left(\ep^{2}a+\ep b\int_{\R^{3}}|\na u|^{2}\right)\De u+V(x)u=f(u)+u^{5}, & u>0 & \text{in }\R^{3},
\end{eqnarray*}
where $V$ and $f$ satisfy similar conditions as that of \cite{He-Zou-2012}.
He, Li and Peng \cite{He-Li-Peng-2014} improved an existence result
of Wang et al. \cite{Wang-Tian-Xu-Zhang-2012} by allowing
that $V$ only satisfies local conditions: there exists a bounded
open set $\Om\subset\R^{3}$ such that
\begin{equation}
\inf_{\Om}V<\inf_{\pa\Om}V\label{eq: local Rabinowitz condition}
\end{equation}
by a penalization method. Later, He and Li \cite{He-Li-2015} proved
the existence of solutions for $\ep$ sufficiently small to the following
problem
\begin{eqnarray}\label{eq: He-Li}
-\left(\ep^{2}a+\ep b\int_{\R^{3}}|\na u|^{2}\right)\De u+V(x)u=u^{q}+u^{5}, & u>0 & \text{in }\R^{3},
\end{eqnarray}
with $V$ satisfying the local condition (\ref{eq: local Rabinowitz condition})
and $1<q<3$. For Kirchhoff problems with more general nonlinearity,
see He \cite{He-2016-JDE}. We remark that all the results mentioned
above were derived by variational methods. In particular, in the case when the subcritical power $q$ belongs to the interval $(1,3)$ as considered in Eq. \eqref{eq: He-Li}, additional difficulty occurs comparing with the case $q\ge 3$. Roughly speaking, this is due to the fact that  the nonlocal term $\left(\int_{\R^{3}}|\na u|^{2}\right)\De u$ is homogeneous of 3-degree, which makes the growth $u^q$ "sublinear" if $q<3$. Thus, the important (AR) condition  fails in this case which prevents from  obtaining a bounded Palais-Smale sequence and using the Nehari manifold directly  to derive solutions. To overcome this difficulty, quite technical methods have been introduced and delicate estimates have been derived in He and Li  \cite{He-Li-2015}.

Quite recently, Li and
the authors of the present paper \cite{LLPWX-2017} established uniqueness
and nondegeneracy results for positive solutions to the unperturbed
Kirchhoff equation
\begin{eqnarray}
-\left(a+b\int_{\R^{3}}|\na u|^{2}\right)\De u+u=u^{p}, & u>0 & \text{in }\R^{3}.\label{eq: unperturbed Kirchhoff}
\end{eqnarray}
Then, using the Lyapunov-Schmidt reduction method, they proved the
existence and uniqueness of single peak solutions to Eq. (\ref{eq: Kirchhoff})
for all $1<p<5$. The building block of the single peak solution obtained
by Li et al. \cite{LLPWX-2017} is the unique positive radial solution
of Eq. (\ref{eq: unperturbed Kirchhoff}). An advantage of this reduction method is that it can deal with the subcritical power $p$ in $(1,5)$ simultaneously, unlike using variational methods as explained in above.

Notice that even though it has been known that problem (\ref{eq: Kirchhoff})
has even multiple single peaks solutions, it is still an open problem  whether
there exist multi-peak solutions to problem (\ref{eq: Kirchhoff}), which
is in striking contrast to the extensive results on multi-peak solutions
to singularly perturbed Schr\"odinger equations (\ref{eq: Schrodinger equations}).
This motivates us to study multi-peak solutions to problem (\ref{eq: Kirchhoff}).
To be precise, we give the definition of multi-peak solutions of Eq.
(\ref{eq: Kirchhoff}) as usual.

\begin{definition}\label{def: multi-peak} Let $k\in \{1,2,\ldots\}$.
We see that $u_{\ep}$
is a $k$-peak solution of \eqref{eq: Kirchhoff} if $u_{\ep}$ satisfies

(i) $u_{\ep}$ has $k$ local maximum points $y_{\ep}^{j}\in\R^{3}$,
$j=1,2,\ldots,k$, satisfying
\[
y_{\ep}^{j}\to a_{j}
\]
for some $a_{j}\in\R^{3}$ as $\ep\to0$ for each $j$;

(ii) For any given $\tau>0$, there exists $R\gg1$, such that
\begin{eqnarray*}
|u_{\ep}(x)|\le\tau &  & \text{for }x\in\R^{3}\backslash\cup_{j=1}^{k}B_{R\ep}(y_{\ep}^{j});
\end{eqnarray*}

(iii) There exists $C>0$ such that
\[
\int_{\R^{3}}(\ep^{2}a|\na u_{\ep}|^{2}+u_{\ep}^{2})\le C\ep^{3}.
\]
\end{definition}

Note that we do not assume $a_{i}\neq a_{j}$ for $i\neq j$. In fact,
there are two  cases for $1\le i,j\le k$: (i) $a_{i}\neq a_{j}$ for all
and $i\neq j$, and (ii) $a_{i}=a_{j}$ for some $i\neq j$. In the present paper,
we will only consider the first case.

To state our main results, we introduce some notation and assumptions.
We assume throughout the paper that $V$ satisfies

(V1) $V\in L^{\wq}(\R^{3})$ and $0<\inf_{\R^{3}}V\le\sup_{\R^{3}}V<\wq$;

(V2) There exist $k$ ($k\ge2$) distinct points $\{a_{1},\ldots,a_{k}\}\subset\R^{3}$
such that for every $1\le i\le k$, $V\in C^{\theta}(\bar{B}_{R_{0}}(a_{i}))$
for some $\ta\in(0,1)$, and
\begin{eqnarray*}
V(a_{i})<V(x) &  & \text{for }0<|x-a_{i}|<r
\end{eqnarray*}
holds for some $r$, $0<r<R_{0}\equiv\frac{1}{2}\min_{1\le i,j\le k,i\ne j}|a_{i}-a_{j}|$.

Denote
\begin{eqnarray*}
\langle u,v\rangle_{\ep}=\int_{\R^{3}}\left(\ep^{2}a\na u\cdot\na v+V(x)uv\right) & \text{and} & \|u\|_{\ep}^{2}=\langle u,u\rangle_{\ep}
\end{eqnarray*}
and let
\[
H_{\ep}=\{u\in H^{1}(\R^{3}):\|u\|_{\ep}<\wq\}.
\]

Our main result reads as follows.

\begin{theorem} \label{thm: main reuslt-existence}Assume that $V$
satisfies (V1) (V2). Then, for $\ep>0$ sufficiently small, equation
\eqref{eq: Kirchhoff} has a $k$-peak solution defined as in the
definition of (\ref{def: multi-peak}) concentrating around $a_{i}$,
$1\le i\le k$.\end{theorem}

To prove Theorem \ref{thm: main reuslt-existence}, let us first recall that to construct multi-peak solutions to the Schr\"odinger equation (\ref{eq: Schrodinger equations}), it is very important
to understand the limiting equation as $\ep\to0$, which is known
as the unperturbed Schr\"odinger equation (\ref{eq: unperturbed Schrodinger equation}).
Denote by $Q_{i}$ the unique (see \cite{Kwong-1989}) positive radial
solution to equation
\begin{eqnarray*}
-\De Q_{i}+V(a_{i})Q_{i}=Q_{i}^{q} &  & \text{in }\R^{n}.
\end{eqnarray*}
Then, to construct a $k$-peak solution to Eq. (\ref{eq: Schrodinger equations})
concentrated at $\{a_{1},\ldots,a_{k}\}$, natural candidates are
functions of the form $u_{\ep}=\sum_{i=1}^{k}Q_{i}((x-y_{i,\ep})/\ep)+\var_{\ep}$,
where $y_{i,\ep}\to a_{i}$ and $\var_{\ep}$ should be appropriately
chosen such that $u_{\ep}$ is indeed a solution to equation (\ref{eq: Schrodinger equations}).

It seemed that the above  idea should also work for problem (\ref{eq: Kirchhoff}) as well, with the unperturbed
Kirchhoff equation (\ref{eq: unperturbed Kirchhoff}) as the limiting
equation. Indeed, to construct single peak solutions to problem (\ref{eq: Kirchhoff}), this idea works, as can be seen in Li et al. \cite{LLPWX-2017}. However, as to construct multi-peak solutions, it turns out to be wrong. That is, there is no
multi-peak solutions of the form $u_{\ep}=\sum_{i=1}^{k}U^{i}((x-y_{\ep}^{i})/\ep)+\var_{\ep}$,
where $U^{i}$ is the unique (see \cite{LLPWX-2017}) positive solution
to equation
\begin{eqnarray*}
-\left(a+b\int_{\R^{3}}|\na u|^{2}\right)\De u+V(a_{i})u=u^{p}, & u>0 & \text{in }\R^{3}.
\end{eqnarray*}
For a proof, see Proposition \ref{prop: nonexistence} in Section \ref{sec: form and location}.
To overcome this difficulty, we will  start from the definition of multi-peak solutions of problem (\ref{eq: Kirchhoff}). We first prove that if $u_{\ep}$ is a $k$-peak solution to \eqref{eq: Kirchhoff},
then $u_{\ep}$ must be of a particular form, and $a_{j}$ must be
critical points of $V$ if $V$ is continuously differentiable in
a neighborhood of $a_{i}$. In fact, via this step, we  prove
that the right limiting equation of problem (\ref{eq: Kirchhoff})
is a system of partial differential equations, see Proposition \ref{prop: the form of multi-peak solutions} in Section \ref{sec: form and location}. This reveals a new
phenomenon of multi-peak solutions for singular perturbation problems, as which is quite different from the known knowledge on singularly perturbed elliptic
equations.

With the help of the above understanding on limiting equations, we
will combine the variational method and the Lyapunov-Schmidt reduction
to prove Theorem \ref{thm: main reuslt-existence}. Note that the
variational functional corresponding to Eq. (\ref{eq: Kirchhoff})
is
\begin{equation}
I_{\epsilon}(u)=\frac{1}{2}\|u\|_{\ep}^{2}+\frac{b\epsilon}{4}\Big(\int_{\R^{3}}|\nabla u|^{2}\Big)^{2}-\frac{1}{p+1}\int_{\R^{3}}u_{+}^{p+1}\label{eq: variational functional}
\end{equation}
 for $u\in H_{\ep}$, where $u_{+}=\max(u,0)$. It is standard to
verify that $I_{\ep}\in C^{2}(H_{\ep})$. So we are left to find a
critical point of $I_{\ep}$. By the results in Section \ref{sec: form and location},
we will construct solutions of the form $u_{\ep}=\sum_{i=1}^{k}w^{i}((x-y_{\ep}^{i})/\ep)+\var_{\ep}$,
where $(w^{1},\cdots,w^{k})$  satisfy the system of  partial
differential equations (see  Proposition \ref{prop: the form of multi-peak solutions} in Section \ref{sec: form and location}). To use reduction method, we will have to prove that the system has a unique and nondegenerate positive solution. Then,
following the scheme of Cao and Peng \cite{Cao-Peng-2009}, we reduce
the problem to find a critical point of a finite dimensional function.  Due to the presence
of the nonlocal term $\left(\int_{\R^{3}}|\na u|^{2}\right)\De u$,
we have to deal with
 the estimates on the orders of $\ep$ carefully, which brings more technical difficulties.

The paper is organized as follows. In section \ref{sec: form and location},
we derive the form and location of multi-peak solutions to Eq. (\ref{eq: Kirchhoff}).
In section 3, we prepare some necessary estimates for the proof of
Theorem \ref{thm: main reuslt-existence}, and in section 4 we prove
Theorem \ref{thm: main reuslt-existence}.

Our notations are standard. Denote $u_{+}=\max(u,0)$ for $u\in\R$.
We use $B_{R}(x)$ (and $\bar{B}_{R}(x)$) to denote open (and close)
balls in $\R^{3}$ centered at $x$ with radius $R$. Without confuse
of notations, we write $\int u$ to denote Lebesgue integrals over
$\R^{3}$, unless otherwise stated.
 By the usual abuse of notations, we write $u(x)=u(r)$
with $r=|x|$ whenever $u$ is a radial function in $\R^{3}$. We
will use $C$ and $C_{j}$ ($j\in\N$) to denote various positive
constants, and $O(t)$, $o(t)$ to mean $|O(t)|\le C|t|$ and $o(t)/t\to0$
as $t\to0$, respectively.

\section{The form and locations of multi-peak solutions\label{sec: form and location}}

In this section, we explore the form of multi-peak solutions of Eq.
(\ref{eq: Kirchhoff}) and locate the related concentrating points.
We will use the following inequality repeatedly.

\begin{lemma} For any $2\le q\le6$, there exists a constant $C>0$
depending only on $V$, $a$ and $q$, but independent of $\ep$,
such that
\begin{equation}
\|\var\|_{L^{q}(\R^{3})}\le C\ep^{\frac{3}{q}-\frac{3}{2}}\|\var\|_{\ep}\label{eq: epsilon-Sobolev inequality}
\end{equation}
holds for all $\var\in H_{\ep}$. \end{lemma}

For a proof of (\ref{eq: epsilon-Sobolev inequality}), see (3.6)
of \cite{LLPWX-2017}.

For convenience, we also introduce notation
\[
u_{\ep,y}(x)=u((x-y)/\ep)
\]
for $\ep>0$ and $y\in\R^{3}$.

Denote by $U^{(i)}\in H^{1}(\R^{3})$ the unique positive radial solution
(see Li et al. \cite{LLPWX-2017}) to equation
\begin{eqnarray*}
-\left(a+b\int|\na u|^{2}\right)\De u+V(a_{i})u=u^{p}, & u>0 & \text{in }\R^{3}.
\end{eqnarray*}
Then, for each $i=1,...,k$, $\ensuremath{U_{\epsilon,y_{i}}^{(i)}}=U^{(i)}((x-y_{i})/\ep)>0$
satisfies
\begin{eqnarray}
-\left(\ep^{2}a+\ep b\int|\na U_{\epsilon,y_{i}}^{(i)}|^{2}\right)\De U_{\epsilon,y_{i}}^{(i)}+V(a_{i})U_{\epsilon,y_{i}}^{(i)}=(U_{\epsilon,y_{i}}^{(i)})^{p} &  & \text{in }\R^{3}.\label{3.2}
\end{eqnarray}
As aforementioned in the introduction, we have

\begin{proposition}\label{prop: nonexistence}
Problem (\ref{eq: Kirchhoff})
has no $k$-peak solutions $(k\ge2)$ of the form
\begin{eqnarray}
u_{\epsilon}(x)=\sum_{i=1}^{k}U_{\epsilon,y_{i}}^{(i)}(x)+\varphi_{\epsilon}(x), & \text{with} & \|\varphi_{\ep}\|_{\epsilon}=o(\epsilon^{\frac{3}{2}}),\label{3.1}
\end{eqnarray}
with
\begin{eqnarray*}
y_{i}\to a_{i} &  & \text{as }\ep\to0
\end{eqnarray*}
for each $i=1,\ldots,k$.
\end{proposition}

We remark that Proposition \ref{prop: nonexistence} holds as well
in the case $a_{1}=a_{2}=\cdots=a_{k}$, provided we assume in addition
that
\[
|y_{i}-y_{j}|/\ep\to\wq
\]
holds for $i\neq j$.

\begin{proof} For simplicity, write
\[
\ep_{1}^{2}=\ep^{2}a+\ep b\int|\na u_{\ep}|^{2}.
\]
Since $U^{(i)}$ and its derivatives decay exponentially at infinity
(see Li et al. \cite{LLPWX-2017}), there exists a constant $\ga>0$,
such that for each $i\neq j$ there hold
\[
\int\left(\ep^{2}\left|\nabla U_{\ep,y_{i}}^{(i)}\cdot\nabla U_{\ep,y_{j}}^{(j)}\right|+U_{\ep,y_{i}}^{(i)}U_{\ep,y_{j}}^{(j)}\right)\D x=O\left(\ep^{3}e^{-\frac{\ga|y_{i}-y_{j}|}{\ep}}\right).
\]
Note that $|y_{i}-y_{j}|/\ep\to\wq$ since we assume $a_{i}\neq a_{j}$.
This implies
\begin{eqnarray}
\int\left(\ep^{2}\left|\nabla U_{\ep,y_{i}}^{(i)}\cdot\nabla U_{\ep,y_{j}}^{(j)}\right|+U_{\ep,y_{i}}^{(i)}U_{\ep,y_{j}}^{(j)}\right)\D x=o(\ep^{3}) &  & \text{for }i\neq j.\label{eq: small interaction}
\end{eqnarray}
Thus,
\begin{equation}
a\ep^{2}\le\ep_{1}^{2}=\epsilon^{2}\Big(a+b\sum_{i=1}^{k}\int|\nabla U^{(i)}|^{2}dx+o_{\ep}(1)\Big)\le A\ep^{2}\label{eq: estimate of epsilon 1}
\end{equation}
for some constant $A>a>0$, where $o_{\ep}(1)\to0$ as $\ep\to0$.

Assume that (\ref{3.1}) gives a solution $u_{\ep}$ to Eq. (\ref{eq: Kirchhoff}).
We derive
\begin{equation}
\sum_{i=1}^{k}\left(-\epsilon_{1}^{2}\Delta U_{\epsilon,y_{i}}^{(i)}+V(x)U_{\epsilon,y_{i}}^{(i)}\right)+(-\epsilon_{1}^{2}\Delta\varphi_{\epsilon}+V(x)\varphi_{\epsilon})=\Big(\sum_{i=1}^{k}U_{\epsilon,y_{i}}^{(i)}+\varphi_{\epsilon}\Big)^{p}.\label{3.3}
\end{equation}
Combining (\ref{3.2}) and (\ref{3.3}) yields
\begin{equation}
\begin{aligned} & \sum_{i=1}^{k}\left(-\left(\epsilon_{1}^{2}-\left(\epsilon^{2}a+\epsilon^{2}b\int|\nabla U^{(i)}|^{2}\right)\right)\Delta U_{\epsilon,y_{i}}^{(i)}\right)+\sum_{i=1}^{k}(V(x)-V(a_{i}))U_{\epsilon,y_{i}}^{(i)}\\
 & \quad+(-\epsilon_{1}^{2}\Delta\varphi_{\epsilon}+V(x)\varphi_{\epsilon})=\Big(\sum_{i=1}^{k}U_{\epsilon,y_{i}}^{(i)}+\varphi_{\epsilon}\Big)^{p}-\sum_{i=1}^{k}(U_{\epsilon,y_{i}}^{(i)})^{p}.
\end{aligned}
\label{3.5}
\end{equation}
Write
\[
K_{i}=\sum_{j\neq i}\int|\na U^{(j)}|^{2}dx>0.
\]
The first term of (\ref{3.5}) can be rewritten as $-\ep^{2}\sum_{i=1}^{k}\left(bK_{i}+o_{\ep}(1)\right)\Delta U_{\epsilon,y_{i}}^{(i)}$.
So regrouping (\ref{3.5}) gives
\begin{equation}
\begin{aligned}-\ep^{2}\sum_{i=1}^{k}\left(bK_{i}+o_{\ep}(1)\right)\Delta U_{\epsilon,y_{i}}^{(i)} & =-\sum_{i=1}^{k}(V(x)-V(a_{i}))U_{\epsilon,y_{i}}^{(i)}-(-\epsilon_{1}^{2}\Delta\varphi_{\epsilon}+V(x)\varphi_{\epsilon})\\
 & \quad+\Big(\big(\sum_{i=1}^{k}U_{\epsilon,y_{i}}^{(i)}+\varphi_{\epsilon}\big)^{p}-\sum_{i=1}^{k}(U_{\epsilon,y_{i}}^{(i)})^{p}\Big).
\end{aligned}
\label{3.6}
\end{equation}

Multiply $U_{\epsilon,y_{j}}^{(j)}$ on both sides of Eq. (\ref{3.6})
and then integrate over $\R^{3}$. By integrating by parts, we obtain
\begin{equation}
\begin{aligned} & \quad\sum_{i=1}^{k}\epsilon^{2}(bK_{i}+o_{\ep}(1))\int\nabla U_{\epsilon,y_{i}}^{(i)}\cdot\nabla U_{\epsilon,y_{j}}^{(j)}\\
 & =-\int\sum_{i=1}^{k}(V(x)-V(a_{i}))U_{\epsilon,y_{i}}^{(i)}U_{\epsilon,y_{j}}^{(j)}-\int(-\epsilon_{1}^{2}\Delta\varphi_{\epsilon}+V(x)\varphi_{\epsilon})U_{\epsilon,y_{j}}^{(j)}\\
 & \quad+\int\left(\left(\sum_{i=1}^{k}U_{\epsilon,y_{i}}^{(i)}+\varphi_{\epsilon}\right)^{p}-\sum_{i=1}^{k}(U_{\epsilon,y_{i}}^{(i)})^{p}\right)U_{\epsilon,y_{j}}^{(j)}\\
 & =:-J_{1}-J_{2}+J_{3}.
\end{aligned}
\label{3.7}
\end{equation}
By (\ref{eq: small interaction}),
\begin{equation}
\sum_{i=1}^{k}\epsilon^{2}(bK_{i}+o_{\ep}(1))\int\nabla U_{\epsilon,y_{i}}^{(i)}\cdot\nabla U_{\epsilon,y_{j}}^{(j)}=\epsilon^{3}\left(bK_{i}\int|\nabla U^{(i)}|^{2}dx+o_{\ep}(1)\right).\label{eq: estimate of LHS}
\end{equation}
To estimate $J_{1}$, split into
\[
\begin{aligned}J_{1} & =\int_{\R^{3}}(V(x)-V(a_{i}))(U_{\epsilon,y_{i}}^{(i)})^{2}+\sum_{i\neq j}\int(V(x)-V(a_{i}))U_{\epsilon,y_{i}}^{(i)}U_{\epsilon,y_{j}}^{(j)}\\
 & =:J_{11}+J_{12}.
\end{aligned}
\]
Since $V$ is bounded, (\ref{eq: small interaction}) implies $J_{12}=o(\ep^{3})$.
Decompose $J_{11}$ into
\[
\begin{aligned}J_{11} & =\int_{\R^{3}}(V(x)-V(y_{i}))(U_{\epsilon,y_{i}}^{(i)})^{2}+\int_{\R^{3}}(V(y_{i})-V(a_{i}))(U_{\epsilon,y_{i}}^{(i)})^{2}\\
 & =:J_{111}+J_{112}.
\end{aligned}
\]
By (V2), we have
\[
\begin{aligned}|J_{111}| & \leq\int_{B_{1}(y_{i})}|V(x)-V(y_{i})|(U_{\epsilon,y_{i}}^{(i)})^{2}+\int_{B_{1}^{c}(y_{i})}|V(x)-V(y_{i})|(U_{\epsilon,y_{i}}^{(i)})^{2}\\
 & \leq C\int_{B_{1}(y_{i})}|x-y_{i}|^{\theta}(U_{\epsilon,y_{i}}^{(i)})^{2}+2\|V\|_{L^{\infty}(\R^{3})}\int_{B_{1}^{c}(y_{i})}(U_{\epsilon,y_{i}}^{(i)})^{2}\\
 & \leq C\epsilon^{3+\theta}\int_{B_{\frac{1}{\epsilon}}(0)}|z|(U^{(i)}(z))^{2}+C\epsilon^{3}\int_{B_{\frac{1}{\epsilon}}^{c}(0)}e^{-2\sigma_{0}|z|}\\
 & =o(\epsilon^{3}).
\end{aligned}
\]
Since $y_{i}\to a_{i}$, we also have
\[
J_{112}=\int(V(y_{i})-V(a_{i}))(U_{\epsilon,y_{i}}^{(i)})^{2}=o(\ep^{3}).
\]
Hence $J_{11}=o(\ep^{3})$, which together with the estimate of $J_{12}$
gives
\begin{equation}
J_{1}=o(\ep^{3}).\label{eq: estimate of J1}
\end{equation}

The estimate of $J_{2}$ follows from (\ref{eq: estimate of epsilon 1})
and H\"older's inequality:
\[
J_{2}=\int(-\epsilon_{1}^{2}\Delta\varphi_{\epsilon}+V(x)\varphi_{\epsilon})U_{\epsilon,y_{j}}^{(j)}=O\left(\|\var_{\ep}\|_{\ep}\left\Vert U_{\epsilon,y_{j}}^{(j)}\right\Vert _{\ep}\right).
\]
Thus, by the assumption $\|\var_{\ep}\|_{\ep}=o(\ep^{3/2})$, we have
\begin{equation}
J_{2}=o(\ep^{3}).\label{eq: estimate of J2}
\end{equation}

To estimate the last term $J_{3}$, we apply an elementary inequality
to obtain
\[
\begin{aligned}|J_{3}| & \le C\int\left(\left(\sum_{i=1}^{k}U_{\epsilon,y_{i}}^{(i)}\right)^{p-1}|\varphi_{\epsilon}|+\sum_{i=1}^{k}U_{\epsilon,y_{i}}^{(i)}|\varphi_{\epsilon}|^{p-1}+|\var_{\ep}|^{p}\right)U_{\epsilon,y_{j}}^{(j)}\\
 & \le C\int\left(\sum_{i=1}^{k}\big(U_{\epsilon,y_{i}}^{(i)}\big)^{p-1}|\varphi_{\epsilon}|+\sum_{i=1}^{k}U_{\epsilon,y_{i}}^{(i)}|\varphi_{\epsilon}|^{p-1}+|\var_{\ep}|^{p}\right)U_{\epsilon,y_{j}}^{(j)}.
\end{aligned}
\]
Using H\"older's inequality, (\ref{eq: epsilon-Sobolev inequality})
and the assumption $\|\var_{\ep}\|_{\ep}=o(\ep^{3/2})$, we derive
\[
\int_{\R^{3}}(U_{\epsilon,y_{i}}^{(i)}\big)^{p-1}U_{\epsilon,y_{j}}^{(j)}|\varphi_{\epsilon}|\le\left\Vert U_{\epsilon,y_{i}}^{(i)}\right\Vert _{L^{p+1}(\R^{3})}^{p-1}\left\Vert U_{\epsilon,y_{j}}^{(j)}\right\Vert _{L^{p+1}(\R^{3})}\|\var_{\ep}\|_{L^{p+1}(\R^{3})}=o(\ep^{3}),
\]
\[
\int_{\R^{3}}U_{\epsilon,y_{i}}^{(i)}U_{\epsilon,y_{j}}^{(j)}|\varphi_{\epsilon}|^{p-1}\le\left\Vert U_{\epsilon,y_{i}}^{(i)}\right\Vert _{L^{p+1}(\R^{3})}\left\Vert U_{\epsilon,y_{j}}^{(j)}\right\Vert _{L^{p+1}(\R^{3})}\|\var_{\ep}\|_{L^{p+1}(\R^{3})}^{p-1}=o(\ep^{3}),
\]
\[
\int_{\R^{3}}|\var_{\ep}|^{p}U_{\epsilon,y_{j}}^{(j)}\le\left\Vert U_{\epsilon,y_{j}}^{(j)}\right\Vert _{L^{p+1}(\R^{3})}\|\var_{\ep}\|_{L^{p+1}(\R^{3})}^{p}=o(\ep^{3}).
\]
Therefore,
\begin{equation}
J_{3}=o(\ep^{3}).\label{eq: estimate of J3}
\end{equation}

Finally, combining (\ref{eq: estimate of LHS}), (\ref{eq: estimate of J1}),
(\ref{eq: estimate of J2}) and (\ref{eq: estimate of J3}) yields
\[
\epsilon^{3}\left(bK_{i}\int_{\R^{3}}|\nabla U^{(i)}|^{2}dx+o_{\ep}(1)\right)=o(\ep^{3})
\]
as $\ep\to0$. This is impossible since $K_{i}>0$. The proof of Proposition
\ref{prop: nonexistence} is complete.
\end{proof}

In the rest of this section, we deduce the form of multi-peak solutions
of Eq. \eqref{eq: Kirchhoff} and locate the corresponding concentrating
points. First we have

\begin{proposition}\label{prop: the form of multi-peak solutions}
Let $u_{\ep}$ be a $k$-peak solution of Eq. \eqref{eq: Kirchhoff}
defined as in the Definition \ref{def: multi-peak}, with local
maximum points at $y_{\ep}^{i}$ and $y_{\ep}^{i}\to a_{i}$. Then,
for $\ep>0$ sufficiently small, $u_{\ep}$ is of the form
\begin{equation}
u_{\epsilon}(x)=\sum_{i=1}^{k}w^{i}((x-y_{\ep}^{i})/\ep)+\varphi_{\epsilon}(x),\label{3.1-1}
\end{equation}
satisfying

(1) $(w^{1},\cdots,w^{k})$ is the unique positive radial solution
to the system
\begin{equation}
\begin{cases}
-(a+b\sum_{i=1}^{k}\int|\na w^{i}|^{2})\De w^{i}+V(a_{i})w^{i}=(w^{i})^{p} & \text{in }\R^{3},\\
w^{i}>0 & \text{in }\R^{3},\\
w^{i}\in H^{1}(\R^{3});
\end{cases}\label{eq: our limit system}
\end{equation}

(2) there holds
\[
\|\varphi_{\ep}\|_{\epsilon}=o(\epsilon^{\frac{3}{2}}).
\]
\end{proposition}

We remark that in the case $a_{1}=\cdots=a_{k}$, the following proof
also implies that
\begin{eqnarray*}
|y_{\ep}^{i}-y_{\ep}^{j}|/\ep\to\infty &  & \text{as }\ep\to0\text{ for }i\neq j.
\end{eqnarray*}

To prove Proposition \ref{prop: the form of multi-peak solutions},
we will need some observations on the system (\ref{eq: our limit system}).

\begin{proposition} \label{prop: uniq and nondege}
For every $a,b>0$,
there exists a unique solution $(w^{1},\cdots,w^{k})$ to the system
(\ref{eq: our limit system}) up to translations. Moreover, the constant
$a+b\sum_{i=1}^{k}\int|\na w^{i}|^{2}$ depends only on $a,b,k$ and
$V(a_{i})$ ($1\le i\le k$), but independent of the choice of the
solutions $w^{i}$, $1\le i\le k$.

Furthermore, for each $1\le i\le k$, $w^{i}$ is nondegenerate in
$H^{1}(\R^{3})$ in the sense that
\[
\operatorname{Ker}\L_{+}^{i}={\rm span}\left\{ \frac{\pa w^{i}}{\pa x_{j}}:1\le j\le3\right\} ,
\]
where $\L_{+}^{i}:H^{1}(\R^{3})\to H^{1}(\R^{3})$ is defined as
\[
\L_{+}^{i}\var\equiv-\left(a+b\sum_{l=1}^{k}\int|\na w^{l}|^{2}\right)\De\var-2b\left(\int\na w^{i}\cdot\na\var\right)\De w^{i}+V(a_{i})\var-p(w^{i})^{p-1}\var
\]
for $\var\in H^{1}(\R^{3})$.
\end{proposition}

To prove this proposition, denote by $Q^{i}$ the unique positive
radial solution to equation
\begin{eqnarray}
-\De u+V(a_{i})u=u^{p}, & u>0 & \text{in }\R^{3}\label{eq: Schrodinger equaiton}
\end{eqnarray}
and $Q^{i}\in H^{1}(\R^{3})$. It is straightforward to deduce from
Kwong \cite{Kwong-1989} that $Q^{i}$ is nondegenerate in $H^{1}(\R^{3})$
in the sense that
\[
\operatorname{Ker}{\cal A}_{+}^{i}={\rm span}\left\{ \frac{\pa Q^{i}}{\pa x_{j}}:1\le j\le3\right\} ,
\]
where ${\cal A}_{+}^{i}$ is the linear operator around $Q^{i}$ defined
as
\[
{\cal A}_{+}^{i}\var\equiv-\De\var+V(a_{i})\var-p(Q^{i})^{p-1}\var
\]
for $\var\in L^{2}(\R^{3})$. Now we prove Proposition \ref{prop: uniq and nondege} briefly .

\begin{proof}[Proof of Proposition \ref{prop: uniq and nondege}]
Denote
\[
c=a+b\sum_{l=1}^{k}\int|\na w^{l}|^{2}.
\]
Then $\bar{w}^{i}(x)=w^{i}(\sqrt{c}x)$, $i=1,\ldots,k$,  satisfy Eq.
(\ref{eq: Schrodinger equaiton}). Hence, the uniqueness result of
Kwong \cite{Kwong-1989} implies that $\bar{w}^{i}(x)=Q^{i}(x-x_{i})$
for some $x_{i}\in\R^{3}$. Therefore,
\[
w^{i}(x)=Q^{i}((x-x_{i})/\sqrt{c}).
\]
This yields
\[
\int|\na w^{i}|^{2}=\sqrt{c}\int|\na Q^{i}|^{2}.
\]
As a consequence,
\[
c=a+\left(b\sum_{i=1}^{k}\|\na Q^{i}\|_{2}^{2}\right)\sqrt{c},
\]
which gives
\[
\sqrt{c}=\frac{1}{2}\left(\bar{b}+\sqrt{\bar{b}^{2}+4a}\right),
\]
where $\bar{b}=b\sum_{i=1}^{k}\|\na Q^{i}\|_{2}^{2}$. This shows
that $c$ depends only on $a,b,V(a_{i})$ and $p$, but independent
of the choice of the solutions $w^{i}$, $1\le i\le k$. Hence,
\[
w^{i}(x)=Q^{i}\left(\frac{2(x-x_{i})}{\bar{b}+\sqrt{\bar{b}^{2}+4a}}\right),\quad1\le i\le k,x_{i}\in\R^{3}
\]
give all the solutions to the system (\ref{eq: our limit system}).
So follows the uniqueness result in Proposition \ref{prop: uniq and nondege}.

Since we have proved that $c$ is a solution-independent positive
constant, the nondegeneracy of $w^{i}$ can be proved by the same
argument as that of Li et al. \cite{LLPWX-2017}. We omit the details.
\end{proof}

Note that since $Q^{i}(x)$ decays exponentially at infinity, we infer
that
\begin{equation}
\max_{1\le i\le k}\left(w^{i}(x)+|\na w^{i}(x)|\right)=O(e^{-\si|x|})\label{eq: exponential decay}
\end{equation}
for some $\si>0$.

Now we can prove Proposition \ref{prop: the form of multi-peak solutions}.

\begin{proof}[Proof of Proposition \ref{prop: the form of multi-peak solutions}]First
recall that, in the case of Schr\"odinger equations (i.e., $b=0$),
if $u_{\ep}$ is a multi-peak solution, then $u_{\ep}$ must be of
the form
\[
u_{\ep}(x)=\sum_{i=1}^{k}U_{\ep,y_{\ep}^{i}}^{i}+\var_{\ep},
\]
where $U^{i}\in H^{1}(\R^{3})$ is the unique positive radial solution
to the equation
\begin{eqnarray*}
-a\De v+V(a_{i})v=v^{p}, & v>0 & \text{in }\R^{3},
\end{eqnarray*}
and $y_{\ep}^{i}$, $\var_{\ep}$ satisfy the listed properties in
Proposition \ref{prop: the form of multi-peak solutions}.

In our case, suppose $u_{\ep}$ is a multi-peak solution to Eq. (\ref{eq: Kirchhoff})
with local maximum points $y_{\ep}^{i}$ ($1\le i\le k$). It is direct
to verify that, for each $1\le i\le k$, $\bar{u}_{\ep}(x)\equiv u_{\epsilon}(\ep x+y_{\ep}^{i})$
is a uniformly bounded sequence in $H^{1}(\R^{n})$ with respect to
$\ep$ and satisfies
\[
-\left(a+b\int|\na\bar{u}_{\ep}|^{2}\right)\De\bar{u}_{\ep}+V(\ep x+y_{\ep}^{i})\bar{u}_{\ep}=\bar{u}_{\ep}^{p}
\]
in $\R^{3}$. So, there exists a subsequence $\ep_{l}\to0$ such that
$\bar{u}_{l}(x)\equiv u_{\ep_{l}}(\ep_{l}x+y_{\ep_{l}}^{i})$ converges
weakly to a function $w^{i}$ in $H^{1}(\R^{3})$ and
\[
a+b\int|\na\bar{u}_{l}|^{2}\to A
\]
as $l\to\wq$ for some constant $A>0$. Then, $w^{i}$ must satisfy
the Schr\"odinger equation
\begin{eqnarray*}
-A\De w^{i}+V(a_{i})w^{i}=(w^{i})^{p} &  & \text{in }\R^{3}.
\end{eqnarray*}
Note that $x=0$ is a maximum point of $w^{i}$. Hence $w^{i}(x)=w^{i}(|x|)$
must be the unique positive radial solution to the above equation.
Moreover, it is well known that $w^{i}(r)=w^{i}(|x|)$ is strictly
decreasing as $|x|\to\wq$. So we can use the same concentrating compactness
arguments as that of multi-peak solutions to Schr\"odinger equations,
to find that
\[
\bar{u}_{l}=\sum_{i=1}^{k}w^{i}\left(\left(x-y_{\ep_{l}}^{i}\right)/\ep_{l}\right)+\var_{\ep_{l}}
\]
 with $y_{\ep_{l}}^{i}$ and $\var_{\ep_{l}}$ satisfying the
properties mentioned in Proposition \ref{prop: the form of multi-peak solutions}.

Finally, noting that $|y_{\ep_{l}}^{i}-y_{\ep_{l}}^{j}|/\ep_{l}\to\wq$
implies that $w_{\ep_{l},y_{\ep_{l}}^{i}}^{i}$, $1\le i\le k$, are
mutually asymptotically orthogonal. That is,
\begin{eqnarray*}
\int\left(\na w_{\ep_{l},y_{\ep_{l}}^{i}}^{i}\cdot\na w_{\ep_{l},y_{\ep_{l}}^{j}}^{j}+w_{\ep_{l},y_{\ep_{l}}^{i}}^{i}\cdot w_{\ep_{l},y_{\ep_{l}}^{j}}^{j}\right)\to0 &  & \text{as }l\to\wq\text{ for }i\neq j.
\end{eqnarray*}
Hence, we deduce
\[
A=\lim_{l\to\wq}\left(a+b\int|\na\bar{u}_{l}|^{2}\right)=a+b\sum_{i=1}^{k}\int|\na w^{i}|^{2}.
\]
Thus, $w^{i}$ ($1\le i\le k$) satisfies the system (\ref{eq: our limit system}).
Then, Proposition \ref{prop: uniq and nondege} implies that the constant
$A$ is independent of the choice of the weak convergent sequence
$\{u_{\ep_{l}}\}$. This in turn means that the above analysis applies
to the whole sequence $\{u_{\ep}\}$. The proof of Proposition \ref{prop: the form of multi-peak solutions}
is complete.\end{proof}

Next we apply the following type of Pohozaev identity to locate multi-peak
solutions of Eq. \eqref{eq: Kirchhoff}.

\begin{proposition} \label{prop: Pohozaev identity}Let $u$ be a
positive solution of Eq. (\ref{eq: Kirchhoff}). Let $\Om$ be a bounded
smooth domain in $\R^{3}$. Then, for each $j=1,2,3$, there hold
\begin{equation}
\begin{aligned}\int_{\Om}\frac{\pa V}{\pa x_{j}}u^{2} & =\left(\ep^{2}a+\ep b\int_{\R^{3}}|\na u|^{2}\right)\int_{\pa\Om}\left(|\na u|^{2}\nu_{j}-2\frac{\pa u}{\pa\nu}\frac{\pa u}{\pa x_{j}}\right)\\
 & \quad+\int_{\pa\Om}Vu^{2}\nu_{j}-\frac{2}{p+1}\int_{\pa\Om}u^{p+1}\nu_{j}.
\end{aligned}
\label{eq: Pohozaev}
\end{equation}
Here $\nu=(\nu_{1},\nu_{2},\nu_{3})$ is the unit outward normal of
$\pa\Om$.\end{proposition}

The proof is obtained by multiplying both sides of Eq. (\ref{eq: Kirchhoff})
by $\pa_{x_{j}}u$ for each $1\le j\le3$ and then integrating
by parts. We omit the details, see Cao-Li-Luo \cite{Cao-Li-Luo-2015}.

\begin{lemma}
 Suppose $V$ satisfies (V1) and $V\in C^{1}(\R^{3})$.
Let $u_{\ep}=\sum_{i=1}^{k}w_{\ep,y_{\ep}^{i}}^{i}+\var_{\ep}$ be
a multi-peak solution to Eq. (\ref{eq: Kirchhoff}) given by Proposition
\ref{prop: the form of multi-peak solutions}. Then $\na V(a_{i})=0$
for each $i=1,\ldots,k$.
\end{lemma}

Note that as a consequence of the above lemma, we find that if Eq.
(\ref{eq: Kirchhoff}) has a concentrating solution, then $V$ must
have at least one critical point.
\begin{proof}
We only prove the result for $i=1$. We use a contradiction argument.
Assume, with no loss of generality, that
\begin{equation}
|V_{x_{1}}(a_{1})|=C_{0}>0.\label{eq: contradiction assumptoin}
\end{equation}
We will apply the Pohozaev identity to $u_{\ep}$ with $\Om=B_{r}(a_{1})$
to deduce the contradiction.

We choose the radius $r$ as follows. Let $r_{0}\equiv\min_{i\ne1}\{1,|y_{i}-y_{1}|/10\}$.
By (\ref{eq: epsilon-Sobolev inequality}) and using the assumption
$\|\var_{\ep}\|_{\ep}=o(\ep^{3/2})$, we have
\[
\|\var_{\ep}\|_{L^{p+1}(\R^{3})}\le C\ep^{\frac{3}{p+1}-\frac{3}{2}}\|\var_{\ep}\|_{\ep}=o(\ep^{3/(p+1)}).
\]
Set $f=\ep^{2}|\na\var_{\ep}|^{2}+|\var_{\ep}|^{2}+|\var_{\ep}|^{p+1}$.
Using polar coordinates, $\int_{0}^{r_{0}}\int_{\pa B_{r}(a_{1})}f=\int_{B_{r_{0}}(a_{1})}f$,
we can choose $r\in(0,r_{0})$ such that
\begin{equation}
\int_{\pa B_{r}(a_{1})}\left(\ep^{2}|\na\var_{\ep}|^{2}+|\var_{\ep}|^{2}+|\var_{\ep}|^{p+1}\right)=o(\ep^{3}).\label{eq: choosing r}
\end{equation}

Now we apply the Pohozaev identity to $u_{\ep}$ with $\Om=B_{r}(a_{1})$
with $r$ being chosen in the above. We obtain
\begin{equation}
\begin{aligned}\int_{B_{r}(a_{1})}\frac{\pa V}{\pa x_{1}}u_{\ep}^{2} & =\ep_{1}^{2}\int_{\pa B_{r}(a_{1})}\left(|\na u_{\ep}|^{2}\nu_{1}-2\frac{\pa u_{\ep}}{\pa\nu}\frac{\pa u_{\ep}}{\pa x_{1}}\right)\\
 & \quad+\int_{\pa B_{r}(a_{1})}Vu_{\ep}^{2}\nu_{1}-\frac{2}{p+1}\int_{\pa B_{r}(a_{1})}u_{\ep}^{p+1}\nu_{1},
\end{aligned}
\label{eq: 2.2}
\end{equation}
where
\[
\ep_{1}^{2}=\ep^{2}a+\ep b\int_{\R^{3}}|\na u_{\ep}|^{2}=O(\ep^{2}).
\]

We estimate (\ref{eq: 2.2}) term by term. To estimate $\int_{B_{r}(a_{1})}\frac{\pa V}{\pa x_{1}}u_{\ep}^{2}$,
split into
\begin{equation}
\begin{aligned}\int_{B_{r}(a_{1})}\frac{\pa V}{\pa x_{1}}u_{\ep}^{2} & =\int_{B_{r}(a_{1})}\left(V_{x_{1}}(x)-V_{x_{1}}(a_{1})\right)u_{\ep}^{2}+V_{x_{1}}(a_{1})\int_{B_{r}(a_{1})}u_{\ep}^{2}.\end{aligned}
\label{eq: Pohozaev identity applied}
\end{equation}
By continuity, we have
\[
\begin{aligned}\left|\int_{B_{r}(a_{1})}\left(V_{x_{1}}(x)-V_{x_{1}}(a_{1})\right)u_{\ep}^{2}\right| & \le\max_{x\in B_{r}(a_{1})}\left|V_{x_{1}}(x)-V_{x_{1}}(a_{1})\right|\int_{B_{r}(a_{1})}u_{\ep}^{2}\end{aligned}
.
\]
By (\ref{eq: exponential decay}), there exists a constant $\ga>0$
which is independent of $\ep$ such that for $i\neq j$
\begin{equation}
\int_{\R^{3}}w_{\ep,y_{\ep}^{i}}^{i}w_{\ep,y_{\ep}^{j}}^{j}=O\left(\ep^{3}e^{-\frac{\ga|y_{i}-y_{j}|}{\ep}}\right).\label{eq: integral estimates}
\end{equation}
Noting that $|a_{i}-a_{1}|>2r$ for each $i\neq1$, using the above
estimates and the assumption $\|\var_{\ep}\|_{\ep}=o(\ep^{3/2})$,
we deduce
\[
C_{1}\ep^{3}\le\int_{B_{r}(a_{1})}u_{\ep}^{2}=\int_{B_{r}(a_{1})}\left(w_{\ep,y_{\ep}^{1}}^{1}\right)^{2}+o(\ep^{3})\le C_{2}\ep^{3}
\]
for $\ep$ sufficiently small, where $C_{1},C_{2}>0$ are independent
of $\ep$. Hence, for $\ep$ sufficiently small, there holds
\begin{equation}
\begin{aligned}\left|\int_{B_{r}(a_{1})}\left(V_{x_{1}}(x)-V_{x_{1}}(a_{1})\right)u_{\ep}^{2}\right| & \le C_{2}\max_{x\in B_{r}(a_{1})}\left|V_{x_{1}}(x)-V_{x_{1}}(a_{1})\right|\ep^{3}\end{aligned}
\label{eq:constant C1}
\end{equation}
and
\begin{equation}
|V_{x_{1}}(a_{1})|\int_{B_{r}(a_{1})}u_{\ep}^{2}\ge C_{0}C_{1}\ep^{3}.\label{eq: constant C2}
\end{equation}
Combining the above two estimates and choosing $r$ sufficiently small,
we obtain
\begin{equation}
\left|\int_{B_{r}(a_{1})}\frac{\pa V}{\pa x_{1}}u_{\ep}^{2}\right|\ge\left(C_{0}C_{1}-C_{2}\max_{x\in B_{r}(a_{1})}\left|V_{x_{1}}(x)-V_{x_{1}}(a_{1})\right|\right)\ep^{3}\ge\frac{C_{0}C_{1}}{2}\ep^{3}.\label{eq: estimate for the LHS}
\end{equation}

On the other hand, we have
\begin{equation}
\begin{aligned}I_{2} & \equiv\ep_{1}^{2}\left|\int_{\pa B_{r}(a_{1})}\left(|\na u_{\ep}|^{2}\nu_{i}-2\frac{\pa u_{\ep}}{\pa\nu}\frac{\pa u_{\ep}}{\pa x_{1}}\right)\right|\\
 & \le C\ep^{2}\int_{\pa B_{r}(a_{1})}\left(\sum_{i=1}^{k}\left|\na w_{\ep,y_{\ep}^{i}}^{i}\right|^{2}+|\na\var_{\ep}|^{2}\right)\\
 & \le C\ep^{2}\left(O(\ep^{-\ga/\ep})+o(\ep)\right)=o(\ep^{3}),
\end{aligned}
\label{eq: estimate for I2}
\end{equation}
and
\begin{equation}
\begin{aligned}I_{3} & \equiv\left|\int_{\pa B_{r}(a_{1})}Vu_{\ep}^{2}\nu_{1}-\frac{2}{p+1}\int_{\pa B_{r}(a_{1})}u_{\ep}^{p+1}\nu_{1}\right|\\
 & \le C\int_{\pa B_{r}(a_{1})}\left(\sum_{i=1}^{k}\left(w_{\ep,y_{\ep}^{i}}^{i}\right)^{2}+|\var_{\ep}|^{2}+\sum_{i=1}^{k}\left(w_{\ep,y_{\ep}^{i}}^{i}\right)^{p+1}+|\var_{\ep}|^{p+1}\right)\\
 & =O(\ep^{-\ga/\ep})+o(\ep^{3})=o(\ep^{3}).
\end{aligned}
\label{eq: estimates for I3}
\end{equation}
In both (\ref{eq: estimate for I2}) (\ref{eq: estimates for I3}),
we have used (\ref{eq: choosing r}) and the exponential decay of
$w^{i}$ at infinity.

Finally, combining (\ref{eq: Pohozaev identity applied}) (\ref{eq: estimate for the LHS})
(\ref{eq: estimate for I2}) (\ref{eq: estimates for I3}) we obtain
\begin{eqnarray*}
\frac{C_{0}C_{1}}{2}\ep^{3}\le o(\ep^{3}), &  & \text{as }\ep\to0.
\end{eqnarray*}
We reach a contradiction. The proof is complete.
\end{proof}

\section{Preliminary estimates}

To obtain multi-peak solutions to Eq. (\ref{eq: Kirchhoff}), Proposition
\ref{prop: the form of multi-peak solutions} inspires us to construct
solutions of the form (\ref{3.1-1}). To this end, let $(w^{1},\cdots,w^{k})$
be the unique positive radial solution to the system (\ref{eq: our limit system})
and let
\begin{eqnarray*}
Y=(y^{1},\cdots,y^{k}) & \text{and} & W_{\ep,Y}=\sum_{i=1}^{k}w_{\ep,y^{i}}^{i}.
\end{eqnarray*}
Recall that we assume $a_{i}\neq a_{j}$ for $i\neq j$ in this paper.
Let $0<\de<\min\{|a_{i}-a_{j}|/4:i\neq j\}$ and denote
\[
D_{\de}=\bar{B}_{\de}(a_{1})\times\cdots\times\bar{B}_{\de}(a_{k}).
\]
Note that if $(y^{1},\cdots,y^{k})\in D_{\de}$, then $|y^{j}-y^{j}|\ge|a_{i}-a_{j}|/2\ge2\de$
with $i\ne j$, which implies by (\ref{eq: exponential decay}) that
\begin{eqnarray}
\int\na w_{\ep,y^{i}}^{i}\cdot\na w_{\ep,y^{j}}^{j}+\left(w_{\ep,y^{i}}^{i}\right)^{q}\left(w_{\ep,y^{j}}^{j}\right)^{r}=O(e^{-\ga/\ep}) &  & \text{with }i\ne j\label{eq: small interaction-1}
\end{eqnarray}
for some constant $\ga>0$ for any given $q,r>0$.

To construct solutions to Eq. (\ref{eq: Kirchhoff}) in the form (\ref{3.1-1}),
we will follow the scheme of Cao and Peng \cite{Cao-Peng-2009}, combining
reduction method and variational method. First, define
\[
J_{\ep}(Y,\var)=I_{\ep}(W_{\ep,Y}+\var)
\]
for $Y=(y^{1},\cdots,y^{k})\in\R^{3k}$ and $\var\in H_{\ep}$. Then,
introduce operators $l_{\ep}$, $\L_{\ep}$ and $R_{\ep}$ as follows:
for $\var,\psi\in H_{\ep}$, define
\begin{equation}
\begin{aligned}l_{\ep}(\var) & =\langle I_{\ep}^{\prime}(W_{\ep,Y}),\var\rangle\\
 & =\langle W_{\ep,Y},\var\rangle_{\ep}+\ep b\left(\int|\na W_{\ep,Y}|^{2}\right)\int\na W_{\ep,Y}\cdot\na\var-\int W_{\ep,Y}^{p}\var,
\end{aligned}
\label{eq: def of l-epsilon}
\end{equation}
and
\begin{equation}
\begin{aligned}\langle\L_{\ep}\var,\psi\rangle & =\langle I_{\ep}^{\prime\prime}(W_{\ep,Y})[\var],\psi\rangle\\
 & =\langle\var,\psi\rangle_{\ep}+\ep b\left(\int|\na W_{\ep,Y}|^{2}\right)\int\na\var\cdot\na\psi\\
 & \quad+2\ep b\left(\int\na W_{\ep,Y}\cdot\na\var\right)\left(\int\na W_{\ep,Y}\cdot\na\psi\right)-p\int W_{\ep,Y}^{p-1}\var\psi,
\end{aligned}
\label{eq: def of biliear L-epsilon}
\end{equation}
and
\begin{equation}
R_{\ep}(\var)=J_{\ep}(Y,\var)-J_{\ep}(Y,0)-l_{\ep}(\var)-\frac{1}{2}\langle\L_{\ep}\var,\var\rangle.\label{eq: 2rd order reminder term}
\end{equation}
Note that $R_{\ep}$ belongs to $C^{2}(H_{\ep})$ since so is every
term in the right hand side of (\ref{eq: 2rd order reminder term}).
In this way, we have expansion
\[
J_{\ep}(Y,\var)=J_{\ep}(Y,0)+l_{\ep}(\var)+\frac{1}{2}\langle\L_{\ep}\var,\var\rangle+R_{\ep}(\var),
\]
where $J_{\ep}(Y,0)=I_{\ep}(W_{\ep,Y})$.

For every $\ep,\de>0$ sufficiently small and for every fixed $Y\in D_{\de}$,
we will prove that $J_{\ep}(Y,\cdot):E_{\ep,Y}\to E_{\ep,Y}$ has
a unique critical point $\var_{\ep,Y}\in E_{\ep,Y}$, where
\[
E_{\ep,Y}=\left\{ u\in H_{\ep}:\left\langle u,\pa_{y_{j}^{i}}w_{\ep,y^{i}}^{i}\right\rangle _{\ep}=0\text{ for }i=1,\ldots,k,\,\,j=1,2,3\right\} .
\]
Then, for each $\ep$,$\de$ sufficiently small, we will find a critical
point $Y_{\ep}$ for the function $j_{\ep}:D_{\de}\to\R$ induced
by
\begin{equation}
Y\mapsto j_{\ep}(Y)\equiv J_{\ep}(Y,\var_{\ep,Y}).\label{eq: reduction function}
\end{equation}
This gives a solution $u_{\ep}\equiv W_{\ep,Y_{\ep}}+\var_{\ep,Y_{\ep}}$
to Eq. (\ref{eq: Kirchhoff}) in virtue of the following lemma, which
can be proved in a standard way, see e.g. Bartsch and Peng \cite{Bartsch-Peng-2007}
and Cao and Peng \cite{Cao-Peng-2009}. We leave the details for the
interested readers.

\begin{lemma}\label{lem: reduction} There exist $\ep_{0}>0$, $\de_{0}>0$
satisfying the following property: for any $\ep\in(0,\ep_{0})$ and
$\de\in(0,\de_{0})$, $Y_{\ep}\in D_{\de}$ is a critical point of
the function $j_{\ep}$ define as in (\ref{eq: reduction function})
if and only if
\[
u_{\ep}\equiv W_{\ep,Y_{\ep}}+\var_{\ep,Y_{\ep}}
\]
 is a critical point of $I_{\ep}$.\end{lemma}

In the below we deduce some necessary estimates for later use.

\begin{lemma}\label{lem: estimate for the first order} Assume that
$\ensuremath{V}$ satisfies (V1) (V2). Then, there exists a constant
$\ensuremath{C>0}$, independent of $\ensuremath{\ep,\de}$, such
that for any $\ensuremath{Y\in D_{\delta}}$ there holds
\[
|l_{\ep}(\var)|\le C\ep^{\frac{3}{2}}\left(\ep^{\theta}+\sum_{j=1}^{k}\left(V(y_{\epsilon}^{j})-V(a_{j})\right)\right)\|\var\|_{\ep}
\]
 for $\ensuremath{\var\in H_{\ep}}$. Here, $\ensuremath{\theta}$
denotes the order of H\"older continuity of $\ensuremath{V}$ in
the neighborhood of $\ensuremath{a_{i}}$, $1\le i\le k$. \end{lemma}
\begin{proof}
Since $(w^{1},\cdots,w^{k})$ solves (\ref{eq: our limit system}),
we have
\[
\begin{aligned}l_{\ep}(\var) & =\sum_{j=1}^{k}\int(V(x)-V(a_{j}))w_{\ep,y_{\epsilon}^{j}}^{j}\var-\int\left(W_{\epsilon,Y}^{p}-\sum_{j=1}^{k}(w_{\epsilon,y_{\epsilon}^{j}}^{j})^{p}\right)\varphi\\
 & \quad+\int|\nabla W_{\epsilon,Y}|^{2}\int\nabla W_{\epsilon,Y}\cdot\nabla\varphi-\int\sum_{i=1}^{k}|\nabla w_{\epsilon,y_{\epsilon}^{i}}^{i}|^{2}\int\nabla W_{\ep,Y}\cdot\nabla\varphi\\
 & =:l_{1}-l_{2}+l_{3}.
\end{aligned}
\]
By the same arguments as that of Lemma 3.2 in \cite{LLPWX-2017},
we obtain
\[
\begin{aligned}|l_{1}| & \le\int\sum_{j=1}^{k}|V(x)-V(y_{\epsilon}^{j})|w_{\ep,y_{\epsilon}^{j}}^{j}\var+\int\sum_{j=1}^{k}(V(y_{\epsilon}^{j})-V(a_{j}))w_{\ep,y_{\epsilon}^{j}}^{j}\var\\
 & \le C\ep^{\frac{3}{2}}\Big(\ep^{\theta}+\sum_{j=1}^{k}(V(y_{\epsilon}^{j})-V(a_{j}))\Big)\|\var\|_{\ep}.
\end{aligned}
\]
To estimate $l_{2}$, note that
\[
\begin{aligned}\left|W_{\epsilon,Y}^{p}-\sum_{j=1}^{k}(w_{\epsilon,y_{\epsilon}^{j}}^{j})^{p}\right| & \le\begin{cases}
C\sum_{i\neq j}\left(w_{\ep,y^{i}}^{i}\right)^{p/2}\left(w_{\ep,y^{j}}^{j}\right)^{p/2}|\varphi|, & 1<p\le2,\vspace{0.5mm}\\
C\sum_{i\ne j}\left((w_{\ep,y^{i}}^{i})^{p-1}w_{\epsilon,y_{\epsilon}^{j}}^{j}+w_{\epsilon,y_{\epsilon}^{i}}^{i}(w_{\epsilon,y_{\epsilon}^{j}}^{j})^{p-1}\right)|\varphi|, & 2<p.
\end{cases}\end{aligned}
\]
So, using (\ref{eq: small interaction-1}) and H\"older's inequality
gives
\[
|l_{2}|\le C\ep^{\frac{3}{2}}e^{-\frac{\ga}{\ep}}\|\var\|_{\ep}\le C\ep^{\frac{3}{2}+\theta}\|\var\|_{\ep}
\]
for $\ep>0$ sufficiently small.

To estimate $l_{3}$, using (\ref{eq: small interaction-1}) again
yields
\[
\begin{aligned}l_{3} & =\ep b\left(\sum_{i\neq j}\int\na w_{\ep,y_{\ep}^{i}}^{i}\cdot\na w_{\ep,y_{\ep}^{j}}^{j}\right)\int\na W_{\ep,Y}\cdot\na\var.\\
 & =O\left(\ep^{2}e^{-\ga/\ep}\|\na W_{\ep,Y}\|_{2}\|\var\|_{2}\right)\\
 & = O\left(\ep^{\frac{3}{2}+\theta}\|\var\|_{\ep}\right).
\end{aligned}
\]
Finally, combining the above estimates gives the required estimate.
\end{proof}
Next we give estimates for $R_{\ep}$ (see (\ref{eq: 2rd order reminder term}))
and its derivatives $R_{\ep}^{(i)}$ for $i=1,2$.

\begin{lemma} \label{lem: error estimates}
There exists a constant
$C>0$, independent of $\ep$ and $b$, such that for $i\in\{0,1,2\}$,
there hold
\[
\|R_{\ep}^{(i)}(\var)\|\le C\ep^{-\frac{3(p-1)}{2}}\|\var\|_{\ep}^{p+1-i}+C(b+1)\ep^{-\frac{3}{2}}\left(1+\ep^{-\frac{3}{2}}\|\var\|_{\ep}\right)\|\var\|_{\ep}^{3-i}
\]
 for all $\var\in H_{\ep}$.
 \end{lemma}

\begin{proof}
This lemma can be proved by the same argument as that of Lemma 3.3
in \cite{LLPWX-2017}. We omit the details.
\end{proof}
Next we consider the operator $\L_{\ep}$ defined as in (\ref{eq: def of biliear L-epsilon}).

\begin{proposition}
There exists $\ep_{1},\de_{1}$ and $\rho>0$
such that for all $\ep\in(0,\ep_{1})$, $\de\in(0,\de_{1})$ and all
$Y\in D_{\de}$, there holds
\begin{eqnarray*}
\|\L_{\ep}\var\|_{\ep}\ge\rho\|\var\|_{\ep}, &  & \var\in E_{\ep,Y}.
\end{eqnarray*}
\end{proposition}

\begin{proof}
We use a contradiction argument. Assume, on the contrary, that there
exist $\ep_{n}\to0$, $\de_{n}\to0$ and $Y_{n}=(y_{n}^{1},\cdots,y_{n}^{k})\in D_{\de_{n}}$
and $\var_{n}\in E_{n}\equiv E_{\ep_{n},Y_{n}}$ such that
\begin{eqnarray}
\langle\L_{\ep_{n}}\var_{n},h_{n}\rangle=o_{n}(1)\|\var_{n}\|_{\ep_{n}}\|h_{n}\|_{\ep_{n}}, &  & \forall\:h_{n}\in E_{n}.\label{eq: B.1}
\end{eqnarray}
Since the equality is homogeneous, we may assume, with no loss of
generality, that $\|\var_{n}\|_{\ep_{n}}=\ep_{n}^{3/2}$.

To deduce contradiction, we introduce $\var_{n}^{i_{0}}(x)=\var_{n}(\ep_{n}x+y_{n}^{i_{0}})$
for each $i_{0}=1,\ldots,,k$. Then, in terms of $\var_{n}^{i_{0}}$, (\ref{eq: B.1}) can be written as
\begin{equation}
\begin{aligned}
& \quad\int\left(\na\var_{n}^{i_{0}}\cdot\na g_{n}+V(\ep_{n}x+y_{n}^{i_{0}})\var_{n}^{i_{0}}g_{n}\right)\\
 & +b\left(\int\sum_{i=1}^{k}|\na w^{i}|^{2}+\int\sum_{j\neq i_{0}}\na w^{i_{0}}\cdot\na w^{j}\left(x+\frac{y_{n}^{i_{0}}-y_{n}^{j}}{\ep_{n}}\right)\right)\int\na\var_{n}^{i_{0}}\cdot\na g_{n}\\
 & +2b\int\na\left(W_{\ep_{n},Y_{n}}\left(\frac{x-y_{n}^{i_{0}}}{\ep_{n}}\right)\right)\cdot\na\var_{n}^{i_{0}}\int\na\left(W_{\ep_{n},Y_{n}}\left(\frac{x-y_{n}^{i_{0}}}{\ep_{n}}\right)\right)\cdot\na g_{n}\\
 & -p\int\left(w^{i_{0}}(x)+\sum_{j\ne i_{0}}w^{j}\left(x+\frac{y_{n}^{i_{0}}-y_{n}^{j}}{\ep_{n}}\right)\right)^{p-1}\var_{n}^{i_{0}}g_{n}\\
 & =o_{n}(1)\left(\int|\na g_{n}|^{2}+V(\ep_{n}x+y_{n}^{i_{0}})g_{n}^{2}\right)^{1/2}
\end{aligned}
\label{eq: B.2}
\end{equation}
for every $g_{n}\in\tilde{E}_{n}$, where
\[
\tilde{E}_{n}=\left\{ g_{n}:g_{n}\left(\frac{x-y_{n}^{i_{0}}}{\ep_{n}}\right)\in E_{n},1\le i_{0}\le k\right\} .
\]
Note that, $g_{n}\in\tilde{E}_{n}$ satisfies
\[
\int\left(\na g_{n}\cdot\na\pa_{x_{j}}w^{i_{0}}+V(\ep_{n}x+y_{n}^{i_{0}})g_{n}\pa_{x_{j}}w^{i_{0}}\right)=0
\]
for all $i_{0}\le k$ and all $1\le j\le3$.

Note also that
\[
\int|\na\var_{n}^{i_{0}}|^{2}+V(\ep_{n}x+y_{n}^{i_{0}})(\var_{n}^{i_{0}})^{2}\D x=1,
\]
which implies that $\var_{n}^{i_{0}}$, $n\ge1$, are uniformly bounded
in $H^{1}(\R^{3})$. So we may assume (up to a subsequence) that
\begin{eqnarray*}
 &  & \var_{n}^{i_{0}}\wto\var\quad\text{weakly in }H^{1}(\R^{3})\\
 &  & \var_{n}^{i_{0}}\to\var\quad\text{in }L_{\loc}^{q}(\R^{3})\quad(1\le q<6)\\
 &  & \var_{n}^{i_{0}}\to\var\quad\text{a.e. in }\R^{3}.
\end{eqnarray*}

Now, for any $g\in C_{0}^{\wq}(\R^{3})$, define
\[
g_{n}=g-\sum_{j=1}^{j}\sum_{i=1}^{k}a_{n,i}^{j}\pa_{x_{j}}w^{i}
\]
for suitable chosen $a_{n,i}^{j}\in\R$, $1\le j\le3$ and $1\le i\le k$,
such that $g_{n}\in\tilde{E}_{n}$. Substitute $g_{n}$ into (\ref{eq: B.2})
and send $n\to\wq$. We obtain, by the same argument as that of \cite[Appendix]{Cao-Peng-2009},
that
\[
\langle\L_{+}^{i_{0}}\var,g\rangle-\langle\L_{+}^{i_{0}}\var,\sum_{j=1}^{3}a_{i_{0}}^{j}\pa_{x_{j}}w^{i_{0}}\rangle=0,
\]
where $\L_{+}^{i_{0}}$ is defined as in Proposition \ref{prop: uniq and nondege}.
Since $\L_{+}^{i_{0}}$ is symmetric and $\pa_{x_{j}}w^{i_{0}}\in\operatorname{Ker}\L_{+}^{i_{0}}$,
we deduce
\[
\langle\L_{+}^{i_{0}}\var,g\rangle=0
\]
for all $g\in C_{0}^{\wq}(\R^{3})$. As a result, we obtain $\var\in\operatorname{Ker}\L_{+}^{i_{0}}$.
By Proposition \ref{prop: uniq and nondege} we infer that
\[
\var=\sum_{j=1}^{3}a_{i_{0}}^{j}\pa_{x_{j}}w^{i_{0}}
\]
for some $a_{i_{0}}^{j}\in\R$.

Claim that $\var\equiv0$. Indeed, since
\[
\int\left(\na\var_{n}^{i_{0}}\cdot\na\pa_{x_{j}}w^{i_{0}}+V(\ep_{n}x+y_{n}^{i_{0}})\var_{n}^{i_{0}}
\pa_{x_{j}}w^{i_{0}}\right)=0
\]
for each $j=1,2,3$, sending $n\to\wq$ yields
\[
\sum_{l=1}^{3}a_{i_{0}}^{l}\int\left(\na\pa_{x_{l}}w^{i_{0}}\cdot\na\pa_{x_{j}}w^{i_{0}}
+V(y^{i_{0}})\pa_{x_{l}}w^{i_{0}}\pa_{x_{j}}w^{i_{0}}\right)=0,
\]
which implies $a_{i_{0}}^{l}=0$ for all $1\le l\le3$ since $w^{i_{0}}$
is a radially symmetric function.

Now we deduce contradiction as follows. First note that by taking
$R$ sufficiently large and recalling that $w^{i_{0}}$ decays exponentially,
we have
\[
\int_{B_{R}^{c}(0)}\left(w^{i_{0}}(x)\right)^{p-1}\left(\var_{n}^{i_{0}}\right)^{2}\le\frac{1}{4}.
\]
Since $\var\equiv0$, we have $\var_{n}^{i_{0}}\to0$ strongly in
$L^{2}(B_{R}(0))$. Therefore, for $n$ sufficiently large there holds
\[
\begin{aligned}\int W_{\ep_{n},Y_{n}}^{p-1}\var_{n}^{2} & =\ep_{n}^{3}\int W_{\ep_{n},Y_{n}}^{p-1}(\ep_{n}x+y_{n}^{i_{0}})\left(\var_{n}^{i_{0}}\right)^{2}\\
 & =\ep_{n}^{3}\int\left(w^{i_{0}}(x)\right)^{p-1}\left(\var_{n}^{i_{0}}\right)^{2}+o_{n}(1)\ep_{n}^{3}\int\left(\var_{n}^{i_{0}}\right)^{2}\\
 & =\ep_{n}^{3}\int_{B_{R}(0)}\left(w^{i_{0}}(x)\right)^{p-1}\left(\var_{n}^{i_{0}}\right)^{2}+\ep_{n}^{3}\int_{B_{R}^{c}(0)}\left(w^{i_{0}}(x)\right)^{p-1}\left(\var_{n}^{i_{0}}\right)^{2}+o_{n}(1)\ep_{n}^{3}\\
 & \le\frac{1}{2}\ep_{n}^{3}.
\end{aligned}
\]
This yields
\[
o_{n}(1)\|\var_{n}\|_{\ep_{n}}^{2}=\langle\L_{\ep_{n}}\var_{n},\var_{n}\rangle\ge\|\var_{n}\|_{\ep_{n}}^{2}-p\int W_{\ep_{n},Y_{n}}^{p-2}\var_{n}^{2}\ge\frac{1}{2}\|\var_{n}\|_{\ep_{n}}^{2}.
\]
We reach a contradiction. The proof is complete.
\end{proof}

\begin{proposition} \label{prop: reduction map}
There exist $\ensuremath{\ep_{0}>0}$
and $\de_{0}>0$ sufficiently small such that for all $\ensuremath{\ep\in(0,\ep_{0})}$
and $\ensuremath{\de\in(0,\de_{0})}$, there exists a $\ensuremath{C^{1}}$
map $\ensuremath{\var_{\ep}:D_{\epsilon,\delta}\to H_{\ep}}$ with
$\ensuremath{Y\mapsto\var_{\ep,Y}\in E_{\ep,Y}}$ satisfying
\begin{eqnarray*}
\left\langle \frac{\pa J_{\ep}(Y,\var_{\ep,Y})}{\pa\var},\psi\right\rangle _{\ep}=0, &  & \forall\:\psi\in E_{\ep,Y}.
\end{eqnarray*}
 Moreover, we can choose $\ensuremath{\tau\in(0,\theta/2)}$ as small
as we wish, such that
\begin{equation}
\|\var_{\ep,Y}\|_{\ep}\le\ep^{\frac{3}{2}}\left(\ep^{\theta-\tau}+\sum_{j=1}^{2}\left(V(y_{\epsilon}^{j})-V(a_{j})\right)^{1-\tau}\right).\label{eq: a simple estimate}
\end{equation}
\end{proposition}

This proposition can be proved by the same arguments as that of Li
et al.  \cite{LLPWX-2017} with minor modifications. We omit the details.

\section{Proof of the main result}

In this section we prove Theorem \ref{thm: main reuslt-existence}.
First we give the following observation.

\begin{lemma}\label{lem: 4.2}
There holds
\[
\left\langle \L_{\ep}\var,\var\right\rangle =O\left(\|\var\|_{\ep}^{2}\right)
\]
for $\var\in E_{\ep,Y}$.
\end{lemma}

\begin{proof}
The proof is direct and we refer to the proof of Lemma 4.3 in \cite{LLPWX-2017}.
\end{proof}

Now  we prove Theorem \ref{thm: main reuslt-existence}.

\begin{proof}[Proof of Theorem \ref{thm: main reuslt-existence}]
Let $\ep_{0}$ and $\de_{0}$ be defined as in Proposition \ref{prop: reduction map}
and let $\ep<\ep_{0}$. Fix $0<\de<\de_{0}$. Let $Y\mapsto\var_{\ep,Y}$
for $Y\in D_{\de}$ be the map obtained in Proposition \ref{prop: reduction map}.
We will find a critical point for the function $j_{\ep}$
defined as in (\ref{eq: reduction function}) by Lemma \ref{lem: reduction}.
By the Taylor expansion, we have
\[
j_{\ep}(Y)=J(Y,\var_{\ep,Y})
=I_{\ep}(W_{\ep,Y})+l_{\ep}(\var_{\ep,Y})+
\frac{1}{2}\langle\L_{\ep}\var_{\ep,Y},\var_{\ep,Y}\rangle
+R_{\ep}(\var_{\ep,Y}).
\]
We analyze the asymptotic behavior of $j_{\ep}$ with respect to $\ep$
first.

By Proposition \ref{Prop: asymptotics of perturbation of $W_{Y}$}, we have
\[
I_{\ep}(W_{\ep,Y})=C_1\ep^{3}+\ep^{3}\sum_{j=1}^{k}C_{2,j}\left(V(y^{j})-V(a_{j})\right)
+O(\ep^{3+\theta})
\]
for some constants $C_1,C_{2,1},\ldots,C_{2,k}\in\R$. Lemma \ref{lem: estimate for the first order}
and Proposition \ref{prop: reduction map} give
\[
l_{\ep}(\var_{\ep,Y})=O(\ep^{3})\left(\ep^{\theta}
+\sum_{j=1}^{k}\left(V(y^{j})-V(a_{j})\right)\right)
\left(\ep^{\theta-\tau}+\sum_{j=1}^{k}\left(V(y^{j})-V(a_{j})\right)^{1-\tau}\right).
\]
Lemma \ref{lem: 4.2} gives $\langle\L_{\ep}\var_{\ep,Y},\var_{\ep,Y}\rangle=O(\|\var_{\ep,Y}\|_{\ep}^{2})$.
Lemma \ref{lem: error estimates} gives
\[
|R_{\ep}(\var_{\ep,Y})|\le C\left(\ep^{-\frac{3(p-1)}{2}}\|\var_{\ep,Y}\|_{\ep}^{p+1}
+\ep^{-\frac{3}{2}}\|\var_{\ep,Y}\|_{\ep}^{3}\right)=o(1)\|\var_{\ep,Y}\|_{\ep}^{2}
\]
by \eqref{eq: a simple estimate}.
Combining the above estimates yields
\[
\begin{aligned}
j_{\ep}(Y) & =C_1\ep^{3}+\ep^{3}\sum_{j=1}^{k}C_{2,j}
\left(V(y^{j})-V(a_{j})\right)+O(\ep^{3+\theta})\\
 & \quad+O(\ep^{3})\left(\ep^{\theta-\tau}+\sum_{j=1}^{k}\left(V(y^{j})-V(a_{j})\right)^{1-\tau}\right)^{2}.
\end{aligned}
\]

Next consider the minimizing problem
\[
j_{\ep}(Y_{\epsilon})\equiv\inf_{Y\in D_{\de}}j_{\ep}(Y).
\]
We claim that  $Y_{\epsilon}$ is an
interior point of $D_{\de}.$

To prove the claim, we apply a comparison argument. Let $e^{j}\in\R^{3} (j=1,\ldots, k)$
with $|e_{j}|=1$, $e_i\neq e_j$  for $i\neq j$ and $\eta>1$. We will choose $\eta>1$ to be sufficiently large. Let $z^{j}_{\epsilon}=a_{j}+\epsilon^{\eta}e_{j}$ such that $Z_{\epsilon}=(z^{1}_{\ep},\cdots,z^{k}_{\ep})\in D_{\delta}$
for a sufficiently large $\eta>1$. By the above asymptotical formula,
we have
\[
\begin{aligned}
j_{\ep}(Z_{\epsilon}) & =C_1\ep^{3}+\ep^{3}\sum_{j=1}^{k}C_{2,j}\left(V(z^{j}_{\ep})-V(a_{j})\right)+O(\ep^{3+\theta})\\
 & \quad+O(\ep^{3})\left(\ep^{\theta-\tau}
 +\sum_{j=1}^{k}\left(V(z^{j}_{\ep})-V(a_{j})\right)^{1-\tau}\right)^{2}.
\end{aligned}
\]
Applying the H\"older continuity of $V$, we derive that
\[
\begin{aligned}j_{\ep}(Z_{\ep}) & =C_1\ep^{3}+O(\ep^{3+\theta\eta})+O(\ep^{3+\theta})+O(\ep^{3}(\ep^{2(\theta-\tau)}+\ep^{2\eta\theta(1-\tau)}))\\
 & =C_1\ep^{3}+O(\ep^{3+\theta}),
\end{aligned}
\]
where $\eta>1$ is chosen to be sufficiently large accordingly. Note
that we also used the fact that $\tau\ll\theta/2$. Thus, by using $j(Y_{\ep})\le j(Z_{\ep})$
we deduce
\[
\ep^{3}\sum_{j=1}^{k}C_{2,j}\left(V(y^{j}_{\ep})-V(a_{j})\right)+O(\ep^{3})\left(\ep^{\theta-\tau}
+\sum_{j=1}^{k}\left(V(y^{j}_{\ep})-V(a_{j})\right)^{1-\tau}\right)^{2}\le O(\ep^{3+\theta}).
\]
That is,
\begin{equation}
\sum_{j=1}^{k}C_{2,j}\left(V(y^{j}_{\ep})-V(a_{j})\right)
+O(1)\left(\ep^{\theta-\tau}+\sum_{j=1}^{k}\left(V(y^{j}_{\ep})-V(a_{j})\right)^{1-\tau}\right)^{2}\le O(\ep^{\theta}).\label{eq: comparision inequality}
\end{equation}
If $Y_{\epsilon}\in \partial D_{\delta}$, then by the assumption (V2), we have
\[
V(y^{j}_{\ep})-V(a_{j})\ge c_{j}>0,\,\,j=1,\ldots,k\]
for some constants $0<c_{j}$. Thus, by noting that $C_2>0$ from
Proposition \ref{Prop: asymptotics of perturbation of $W_{Y}$} and sending
$\ep\to0$, we infer from (\ref{eq: comparision inequality}) that
\[
0<\sum_{j=1}^k C_{2,j}c_{j}\le0.
\]
We reach a contradiction. This proves the claim.
Thus $Y_{\ep}=(y^{1}_{\epsilon},\ldots,y^{k}_{\epsilon})$ is
a critical point of $j_{\ep}$ in the interior of  $D_{\delta}$.

Theorem \ref{thm: main reuslt-existence} now follows from the claim
and Lemma \ref{lem: reduction}.
\end{proof}

\appendix

\section{asymptotic behaviors }

Let $(w^{1},\cdots,w^{k})$ be the unique positive radial solution
of the system (\ref{eq: our limit system}). Then, for any $\ep>0$
and $Y=(y^{1},\cdots,y^{k})\in D_{\de}$, the functions $w_{\ep,y^{i}}^{j}$,
$1\le j\le k$, satisfy
\begin{eqnarray}
-\left(\ep^{2}a+\ep b\sum_{i=1}^{k}\int|\na w_{\ep,y^{i}}^{i}|^{2}\right)\De w_{\ep,y^{j}}^{j}+V(a_{j})w_{\ep,y^{j}}^{j}=(w_{\ep,y^{j}}^{j})^{p} &  & \text{in }\R^{3}.\label{eq:scaled system}
\end{eqnarray}
\begin{proposition}\label{Prop: asymptotics of perturbation of $W_{Y}$}
Assume that $V$ satisfies (V1) and (V2). Let $Y\in D_{\delta}$ and
$W_{\ep,Y}=\sum_{i=1}^{k}w_{\ep,y^{i}}^{i}$. Then for $\ep>0$ sufficiently
small, we have
\[
I_{\ep}(W_{\ep,Y})=C_{1}\ep^{3}+\sum_{j=1}^{k}C_{2,j}\left(V(y^{j})-V(a_{j})\right)\ep^{3}+O(\ep^{3+\theta}),
\]
where

\[
C_{1}=\Big(\frac{1}{2}-\frac{1}{p+1}\Big)\sum_{j=1}^{k}\int(w^{j})^{p+1}-\frac{b}{4}\left(\sum_{i=1}^{k}\int|\na w^{i}|^{2}\right)^{2}
\]
and
\[
C_{2,j}=\frac{1}{2}\int(w^{j})^{2},\quad j=1,\ldots,k,
\]
and $\theta$ is the H\"older continuity of $V$ in the neighborhood
of $a_{i}$, $1\le i\le k$. \end{proposition}
\begin{proof}
Recall that
\[
I_{\ep}(W_{\ep,Y})=\frac{1}{2}\int\left(\ep^{2}a|\na W_{\ep,Y}|^{2}+V(x)W_{\ep,Y}^{2}\right)+\frac{\ep b}{4}\left(\int|\na W_{\ep,Y}|^{2}\right)^{2}-\frac{1}{p+1}\int W_{\ep,Y}^{p+1}.
\]
Since $a_{i}\neq a_{j}$ for $i\neq j$ and $w^{i}$ ($1\le i\le k$)
decays exponentially at infinity, the estimates (\ref{eq: small interaction-1})
hold for $i\neq j$. Hence, for $\ep>0$ sufficiently small, we have
\begin{equation}
\begin{aligned}I_{\ep}(W_{\ep,Y}) & =\frac{1}{2}\sum_{i=1}^{k}\int\left(\ep^{2}a|\na w_{\ep,y^{i}}^{i}|^{2}+V(x)\left(w_{\ep,y^{i}}^{i}\right)^{2}\right)\\
 & \quad+\frac{\ep b}{4}\left(\sum_{i=1}^{k}\int|\na w_{\ep,y^{i}}^{i}|^{2}\right)^{2}-\frac{1}{p+1}\int\left(\sum_{i=1}^{k}w_{\ep,y^{i}}^{i}\right)^{p+1}+O(e^{-\ga/\ep})
\end{aligned}
\label{eq: A.1}
\end{equation}
for some $\ga>0$. Note that $w_{\ep,y^{i}}^{i}$ satisfies Eq. (\ref{eq:scaled system}).
Thus,
\[
\left(\ep^{2}a+\ep b\sum_{j=1}^{k}\int|\na w_{\ep,y^{j}}^{j}|^{2}\right)\int|\na w_{\ep,y^{i}}^{i}|^{2}+V(a_{i})\left(w_{\ep,y^{i}}^{i}\right)^{2}=\int\left(w_{\ep,y^{i}}^{i}\right)^{p+1}.
\]
Substitute the above identity into (\ref{eq: A.1}). We obtain
\begin{equation}
\begin{aligned}I_{\ep}(W_{\ep,Y}) & =\frac{1}{2}\sum_{i=1}^{k}\int\left(V(x)-V(a_{i})\right)\left(w_{\ep,y^{i}}^{i}\right)^{2}-\frac{\ep b}{4}\left(\sum_{i=1}^{k}\int|\na w_{\ep,y^{i}}^{i}|^{2}\right)^{2}\\
 & \quad+\frac{1}{2}\int\sum_{i=1}^{k}\left(w_{\ep,y^{i}}^{i}\right)^{p+1}-\frac{1}{p+1}\int\left(\sum_{i=1}^{k}w_{\ep,y^{i}}^{i}\right)^{p+1}+O(e^{-\ga/\ep}).
\end{aligned}
\label{eq: A.2}
\end{equation}
Apply the arguments in the appendix of Li et al. \cite{LLPWX-2017},
we have
\begin{equation}
\int(V(x)-V(a_{j}))w_{\ep,y^{j}}^{j}=\ep^{3}(V(y^{j})-V(a_{j}))\int(w^{j})^{3}+O(\ep^{3+\theta}).\label{eq: A.3}
\end{equation}
Apply the following elementary inequalities

\[
\begin{aligned}\big||a+b|^{p}-a^{p}-b^{p}-pa^{p-1}b-pab^{p-1}\big| & \le\left\{ \begin{array}{ll}
C|a||b|^{p-1}, & \text{if }|b|\leq|a|,\\
C|a|^{p-1}|b| & \text{if }|b|\ge|a|,
\end{array}\right.\\
 & \leq C|a|^{\frac{p}{2}}|b|^{\frac{p}{2}},\,\,\,(1<p\leq2)
\end{aligned}
\]
and
\[
\big||a+b|^{p}-a^{p}-b^{p}-pa^{p-1}b-pab^{p-1}\big|\leq|a|^{p-2}|b|^{2}+|a|^{2}|b|^{p-2},\,\,\,(p>2).
\]
We derive
\begin{equation}
\begin{aligned}\int W_{Y}^{p+1} & =\sum_{i=1}^{k}\int\left(w_{\ep,y^{i}}^{i}\right)^{p+1}+(p+1)\int\left((w_{\epsilon,y^{1}}^{1})^{p}w_{\epsilon,y^{2}}^{2}+w_{\epsilon,y^{1}}^{1}(w_{\epsilon,y^{2}}^{2})^{p}\right)\\
 & \quad+\left\{ \begin{array}{ll}
C\int(w_{\epsilon,y^{1}}^{1})^{\frac{p}{2}}(w_{\epsilon,y^{2}}^{2})^{\frac{p}{2}}, & \ensuremath{1<p\leq2},\vspace{0.3cm}\\
C\int[(w_{\epsilon,y^{1}}^{1})^{p-1}(w_{\epsilon,y^{2}}^{2})^{2}(w_{\epsilon,y^{1}}^{1})^{2}(w_{\epsilon,y^{2}}^{2})^{p-1}], & p>2
\end{array}\right.\\
 & =\ep^{3}\sum_{i=1}^{k}\int\left(w^{i}\right)^{p+1}+O(\ep^{3+\theta}).
\end{aligned}
\label{eq: A.4}
\end{equation}
The required estimate follows from (\ref{eq: A.2}) (\ref{eq: A.3})
and (\ref{eq: A.4}).
\end{proof}

\emph{Acknowledgments. } The authors would like to thank Shusen Yan for many useful suggestions during the preparation of this paper.  Some of this research took place during
a one-month stay by the last named author at the Central China Normal University.
He would like to thank the institute for the gracious hospitality
during this time.

\end{document}